\documentclass[12pt]{amsart}
\usepackage{amssymb,latexsym,verbatim}
\theoremstyle{plain}
\newtheorem{thm}{Theorem}[section]
\newtheorem{lem}[thm]{Lemma}
\newtheorem{cor}[thm]{Corollary}
\newtheorem{prop}[thm]{Proposition}

\theoremstyle{definition}
\newtheorem{defn}[thm]{Definition}

\theoremstyle{remark}
\newtheorem{exam}[thm]{Example}
\newtheorem{rem}[thm]{Remark}

\DeclareMathOperator{\Ext}{Ext}
\DeclareMathOperator{\Supp}{Supp}

\DeclareMathOperator{\Hom}{Hom}
\DeclareMathOperator{\Ker}{Ker}
\DeclareMathOperator{\Coker}{Coker}

\DeclareMathOperator{\depth}{depth}

\DeclareMathOperator{\Ass}{Ass}

\DeclareMathOperator{\lc}{H}

\newcommand{\lo}{\longrightarrow}

\newcommand{\fa}{\mathfrak{a}}

\newcommand{\fp}{\mathfrak{p}}

\newcommand{\Z}{\mathbb{Z}}
\begin{document}

\bibliographystyle{amsplain}

\author{Moharram Aghapournahr}
\address{ Moharram Aghapournahr\\Faculty of Mathematical Sciences, Teacher Training
University, 599 Taleghani Avenue, Tehran 15614, Iran.}

\email{m.aghapour@gmail.com}

\author{Leif Melkersson}
\address{Leif Melkersson\\Department of Mathematics, Link\"{o}ping University, S-581 83 Link\"{o}ping, Sweden.}

\email{lemel@mai.liu.se}



\title[Local cohomology]
{Local cohomology: Associated primes,  \\ artinianness and asymptotic behaviour}


\begin{abstract}
 Let $R$ be a noetherian ring, $\fa$ an ideal of $R$, $M$ an $R$--module and $n$ a non-negative integer.
 In this paper we first will study the finiteness properties of the kernel and the cokernel of the natural
 map $f:\Ext^n_{R}(R/\fa,M)\lo \Hom_{R}(R/\fa,\lc^{n}_{\fa}(M))$. Then we will get some corollaries about the
 associated primes and artinianness of local cohomology modules. Finally we will study the asymptotic behaviour
 of the kernel and the cokernel of this natural map in the graded case.

\end{abstract}

\maketitle

\section{Introduction}

Throughout $R$ is a commutative noetherian ring. Our terminology follows the book \cite{BSh} on local cohomology. Huneke
formulated and discussed several problems in \cite{Hu} about local cohomology modules $\lc^{n}_{\fa}(M)$ of $M$. One of them
is when do they have just finitely many associated prime ideals. Another one is about when such modules are artinian. Furthermore
the following conjecture was made by Grothendieck in \cite{SGA2}:

\noindent {\bf Conjecture:} For any ideal $\fa$ and finite $R$--module $M$, the module $\Hom_{R}(R/\fa,\lc^{n}_{\fa}(M))$ is finite
for all $n\geq 0$.

This was thought of  as a substitute for the artinianness of $\lc^{n}_{\fa}(M)$, in the case when $\fa$ is the  maximal ideal of a local
 ring. Morover the finiteness of $\Hom_{R}(R/\fa,\lc^{n}_{\fa}(M))$ implies the finiteness of the set of associated prime ideals of $\lc^{n}_{\fa}(M)$.
  Although this conjecture is not true in general as shown by Hartshorne in \cite{Ha}, there are some attempts to show that under some conditions, for
  some number $n$, the module $\Hom_{R}(R/\fa,\lc^{n}_{\fa}(M))$ is finite, see \cite[Theorem 3.3]{AKS}, \cite[Theorem 6.3.9]{DY1} and \cite[Theorem 2.1]
  {DY2}.

In section 2, in view of Grothendieck's conjecture and \cite[Theorem 6.3.9]{DY1}, we first study finiteness properties of the kernel and the cokernel of
 the natural homomorphism
\begin{center}
$f:\Ext^n_{R}(R/\fa,M)\lo \Hom_{R}(R/\fa,\lc^{n}_{\fa}(M))$
\end{center}
under certain conditions, using the technique introduced by the second author in \cite{Mel}. We will find some relations  between finiteness of the
associated primes of $\Ext^n_{R}(R/\fa,M)$ and $\lc^{n}_{\fa}(M)$ that will give us a generalization of \cite[Proposition 1.1 part (b)]{Mar},
 \cite[Lemma 2.5 part (c)]{ASc} and \cite[Proposition 2.2]{BRS}. Also we show that under certain conditions $\Ext^n_{R}(R/\fa,M)$ is artinian
 if and only if $\lc^{n}_{\fa}(M)$ is. One of our main results of this paper is

\noindent{\bf Corollary 2.2.}
\begin{flushleft}
(a){\it If $\Ext^{n-j}_{R}(R/\fa, \lc^{j}_{\fa}(M))= 0$ for all $j<n$, then $f$ is injective.}
\end{flushleft}
(b){\it If $\Ext^{n+1-j}_{R}(R/\fa, \lc^{j}_{\fa}(M))= 0$ for all $j<n$, then $f$ is surjective.}
\begin{flushleft}
(c){\it If $\Ext^{t-j}_{R}(R/\fa, \lc^{j}_{\fa}(M))= 0$ for $t=n,n+1$ and for all $j<n$, then}
\end{flushleft}

{\it $f$ is an isomorphism.}

In section 3 we will study the asymptotic behaviour of the kernel and the cokernel of the above mentioned natural map in the graded case.


\section{Associated primes and artinianness}

\begin{prop}\label{P:natmap}
Let $R$ be a noetherian ring, $\fa$ an ideal of $R$ and $M$ an
$R$--module. Let $n$ be a non-negative integer and $\mathcal S$ be a
Serre subcategory of the category of $R$--modules i.e. it is closed
under taking submodules, quotients and extensions. Consider the
natural homomorphism
\begin{center}
$f:\Ext^n_{R}(R/\fa,M)\lo \Hom_{R}(R/\fa,\lc^{n}_{\fa}(M))$

\end{center}

  \item[(a)]If $\Ext^{n-j}_{R}(R/\fa, \lc^{j}_{\fa}(M))$ belongs to $\mathcal S$ for all $j<n$, then $\Ker{f}$ belongs to $\mathcal S$. {\rm{(}}In
  particular if $\Ext^{n-j}_{R}(R/\fa, \lc^{j}_{\fa}(M))$ is finite for all $j<n$, then $\Ker{f}$ is finite.{\rm{)}}
  \item[(b)]If $\Ext^{n+1-j}_{R}(R/\fa, \lc^{j}_{\fa}(M))$ belongs to $\mathcal S$ for all $j<n$, then $\Coker{f}$ belongs to $\mathcal S$. {\rm{(}}
  In particular if $\Ext^{n+1-j}_{R}(R/\fa, \lc^{j}_{\fa}(M))$ is finite for all $j<n$, then $\Coker{f}$ is finite.{\rm{)}}
  \item[(c)]If $\Ext^{t-j}_{R}(R/\fa, \lc^{j}_{\fa}(M))$ belongs to $\mathcal S$ for $t=n,n+1$ and for all $j<n$, then $\Ker{f}$ and $\Coker{f}$ both
   belong to $\mathcal S$. Thus $\Ext^n_{R}(R/\fa,M)$ belongs to $\mathcal S$ if and only if $\Hom_{R}(R/\fa,\lc^{n}_{\fa}(M))$ belongs to $\mathcal S$.
    {\rm{(}}In particular if $\Ext^{t-j}_{R}(R/\fa, \lc^{j}_{\fa}(M))$ is finite for $t=n,n+1$ and for all $j<n$ then $\Ker{f}$ and $\Coker{f}$
     both are finite, thus $\Ext^n_{R}(R/\fa,M)$ is finite if and only if $\Hom_{R}(R/\fa,\lc^{n}_{\fa}(M))$ is finite. {\rm{(}}See also
      \cite[Theorem 6.3.9]{DY1}{\rm{)}}{\rm{)}}.

\end{prop}

\begin{proof}
By the exact sequence  $0\rightarrow \Gamma_{\fa}(M)\rightarrow M\rightarrow M/\Gamma_{\fa}(M)\rightarrow 0$  we obtain the exact sequence
(\textasteriskcentered) below:

\begin{center}
$\Ext^n_{R}(R/\fa,\Gamma_{\fa}(M))\rightarrow \Ext^n_{R}(R/\fa,M)\overset{h}\rightarrow \Ext^n_{R}(R/\fa,M/\Gamma_{\fa}(M))\rightarrow
 \Ext^{n+1}_{R}(R/\fa,\Gamma_{\fa}(M)).$
\end{center}

 We have the commutative diagram
$$\begin{matrix}\\
&&\Ext^n_{R}(R/\fa,M) &\overset{h}\lo &\Ext^n_{R}(R/\fa,M/\Gamma_{\fa}(M))\\
&&f\big\downarrow & &\big\downarrow{\bar f}\\
&&\Hom_{R}(R/\fa,\lc^{n}_{\fa}(M)) &\overset{\bar{h}}\lo &\Hom_{R}(R/\fa,\lc^{n}_{\fa}(M/\Gamma_{\fa}(M)))\\
\end{matrix}.$$
Note that $\bar{h}$ is an isomorphism.
Let $g=\bar{h}^{-1}\circ{\bar{f}}$. Thus $f=g\circ{h}$ and we obtain the exact sequence
\begin{flushleft}
$0\rightarrow \Ker{h}\rightarrow \Ker{f}\rightarrow
\Ker{g}\rightarrow \Coker{h}\rightarrow \Coker{f}\rightarrow
\Coker{g}\rightarrow{0}$
\end{flushleft}

{\it Proof of} (a):

For $n=0$, since $\Ext^0_{R}(R/\fa,M) \cong
\Hom_{R}(R/\fa,\Gamma_{\fa}(M))$ thus $\Ker{f}=0$ and the result is
clear. For $n\geq 1$ by hypothesis and the exact sequence
(\textasteriskcentered), $\Ker{h}$ belongs to $\mathcal S$.
Therefore it is enough to show that $\Ker{g}$ belongs to $\mathcal
S$. We prove this by induction on $n$. For $n=1$ since $\bar{h}$ is
an isomorphism, hence $\Ker{g}=\Ker{\bar{f}}$ but $\bar{f}$ is an
isomorphism by \cite[Lemma 7.9]{Mel}. Assume $n>1$ and the case
$n-1$ is settled. Let $E$ be
 the injective hull of $M/\Gamma_{\fa}(M)$ and set $N=E/(M/\Gamma_{\fa}(M))   $. Since $\Gamma_{\fa}(M/\Gamma_{\fa}(M))=0$, we have $\Hom_{R} (R/\fa,E)=0$
 and $\Gamma_{\fa}(E)=0$. Thus we get the isomorphisms
\begin{center}
$\Ext^i_{R}(R/\fa,N)\cong \Ext^{i+1}_{R}(R/\fa,M/\Gamma_{\fa}(M))$
\end{center}
and
\begin{center}
$\lc^{i}_{\fa}(N)\cong \lc^{i+1}_{\fa}(M/\Gamma_{\fa}(M))\cong \lc^{i+1}_{\fa}(M)$
\end{center}
for all $i\geq 0$. Therefore
\begin{center}
$\Hom_{R}(R/\fa,\lc^{n-1}_{\fa}(N))\cong \Hom_{R}(R/\fa,\lc^{n}_{\fa}(M))$.
\end{center}
In addition, for all $j<n-1$ the modules
\begin{center}
$\Ext^{n-1-j}_{R}(R/\fa, \lc^{j}_{\fa}(N))\cong \Ext^{n-1-j}_{R}(R/\fa, \lc^{j+1}_{\fa}(M))$
\end{center}
belong to $\mathcal S$. Now, by the induction hypothesis, the kernel of the natural homomorphism
\begin{center}
$f^\prime: \Ext^{n-1}_{R}(R/\fa,N)\lo \Hom_{R}(R/\fa,\lc^{n-1}_{\fa}(N))$

\end{center}
belongs to $\mathcal S$. But $\Ker{f^\prime}=\Ker{g}$, so $\Ker{g}$ belongs to $\mathcal S$.
\begin{flushleft}
{\it Proof of} (b):
\end{flushleft}
 The proof is similar to (a).
\begin{flushleft}
{\it Proof of} (c):
\end{flushleft}
It is clear by (a) and (b).
\end{proof}

The following corollary is one of our main results of this paper. It is well-known by using a spectral sequence argument as
in \cite[Proposition 1.1 (b)]{Mar} or another method as in \cite[Lemma 2.5 (b)]{ASc} to prove that, (when
$M$ is a finite module and $n=\depth_\fa M$) or more generally when the local cohomology vanishes below $n$, for example
 use \cite[Theorem 11.3]{Rot}, then the natural map
\begin{center}
$f:\Ext^n_{R}(R/\fa,M)\lo \Hom_{R}(R/\fa,\lc^{n}_{\fa}(M))$
\end{center}
is an isomorphism.
 In the following corollary we show that without any condition on $M$ it is not necessary that these local cohomology modules
 are zero in order that the natural map is an isomorphism.

\begin{cor}\label{C:isiso}
  \item[(a)]If $\Ext^{n-j}_{R}(R/\fa, \lc^{j}_{\fa}(M))= 0$ for all $j<n$, then $f$ is injective.
  \item[(b)]If $\Ext^{n+1-j}_{R}(R/\fa, \lc^{j}_{\fa}(M))= 0$ for all $j<n$, then $f$ is surjective.
  \item[(c)]If $\Ext^{t-j}_{R}(R/\fa, \lc^{j}_{\fa}(M))= 0$ for $t=n,n+1$ and for all $j<n$, then

$f$ is an isomorphism.

\end{cor}
\begin{proof}

 In proposition \ref{P:natmap} set $\mathcal S=\{0\}$.
\end{proof}
\begin{cor}\label{C:iso}
If $\lc^{i}_{\fa}(M)=0$ for all $i<n$ {\rm{(}}in particular, when $M$ is finite and $n=\depth_\fa M${\rm{)}} then
\begin{center}
$\Ext^n_{R}(R/\fa,M) \cong \Hom_{R}(R/\fa,\lc^{n}_{\fa}(M))$
\end{center}
\end{cor}

\begin{cor}\label{C:art}
If $\Ext^{t-j}_{R}(R/\fa, \lc^{j}_{\fa}(M))$ is artinian for $t=n,n+1$ and for all $j<n$, then $\Ext^n_{R}(R/\fa,M)$ is artinian
if and only if $\lc^{n}_{\fa}(M)$ is artinian.
\end{cor}
\begin{proof}
It is immediate by using \ref{P:natmap} (c) and \cite[Theorem 7.1.2]{BSh}.
\end{proof}
\begin{cor}\label{C:ass}
  \item[(a)]If $\Ext^{n-j}_{R}(R/\fa, \lc^{j}_{\fa}(M))= 0$ for all $j<n$, then
\begin{center}
   $\Ass_R(\lc^{n}_{\fa}(M))\subset \Ass_R(\Ext^n_{R}(R/\fa,M))\cup \Ass_R(\Coker{f})$.
\end{center}
  \item[(b)]If $\Ext^{n+1-j}_{R}(R/\fa, \lc^{j}_{\fa}(M))= 0$ for all $j<n$, then
\begin{center}
  $\Ass_R(\lc^{n}_{\fa}(M))\subset \Ass_R(\Ext^n_{R}(R/\fa,M))\cup \Supp_R(\Ker{f})$.
\end{center}
  \item[(c)]If $\Ext^{t-j}_{R}(R/\fa, \lc^{j}_{\fa}(M))= 0$ for $t=n,n+1$ and for all $j<n$, then
\begin{center}
$\Ass_R(\Ext^n_{R}(R/\fa,M))= \Ass_R(\lc^{n}_{\fa}(M))$.
\end{center}
\end{cor}

For the proof of part (b) in corollary \ref{C:ass} we need the following lemma.
\begin{lem}
For any exact sequence $N\lo L\lo T\lo 0$ of $R$--modules and $R$--homomorphisms we have $$\Ass_R(T)\subset \Ass_R(L)\cup \Supp_R(N)$$
\end{lem}

\begin{proof}
Let $\fp\in \Ass_R(T)\setminus \Supp_R(N)$ then $N_{\fp}=0$, so $L_{\fp}\cong T_{\fp}$. Now ${\fp}R_{\fp}\in \Ass_{R_\fp}(T_{\fp})$,
hence ${\fp}R_{\fp}\in \Ass_{R_\fp}(L_{\fp})$, so $\fp\in \Ass_R(L)$.
\end{proof}

Divaani-Aazar and Mafi introduced in \cite{DM} {\it weakly Laskerian} modules. An $R$--module $M$ is weakly Laskerian if for any
submodule $N$ of $M$ the quotient $M/N$ has finitely many associated primes. The weakly Laskerian modules form a Serre subcategory
of the category of $R$--modules. In fact it is the largest Serre subcategory such that each module in it has just finitely many associated
prime ideals.
In the following corollary we give some conditions for the finiteness of the sets of associated primes of $\Ext^n_{R}(R/\fa,M)$ and $\lc^{n}_{\fa}(M)$
respectively.
\begin{cor}\label{C:asswl}
 \item[(a)] If $\Ass_R(\lc^{n}_{\fa}(M))$ is a finite set and $\Ext^{n-j}_{R}(R/\fa, \lc^{j}_{\fa}(M))$ is weakly Laskerian for all $j<n $ {\rm{(}}
 in particular if $\lc^{j}_{\fa}(M)$ is weakly Laskerian for all $j<n${\rm{)}}, then $\Ass_R(\Ext^n_{R}(R/\fa,M))$ is a finite set. {\rm{(}}See
 also \cite[ Corollary 6.3.11]{DY1}{\rm{)}}
 \item[(b)] If $\Ext^n_{R}(R/\fa,M)$ and $\Ext^{n+1-j}_{R}(R/\fa, \lc^{j}_{\fa}(M))$ for all $j<n$ are weakly Laskerian {\rm{(}}in particular
 if $M$ and $\lc^{j}_{\fa}(M)$ for all $j<n$ are weakly Laskerian{\rm{)}}, then $\Ass_R(\lc^{n}_{\fa}(M))$ is a finite set. {\rm{(}}Compare
 with \cite [ Corollary 2.7]{DM}{\rm{)}}
\end{cor}
\begin{proof}
Note that $\Ass_R(\lc^{n}_{\fa}(M))=\Ass_R(\Hom_{R}(R/\fa,\lc^{n}_{\fa}(M)))$. We use the exact sequence
\begin{center}
$0\rightarrow \Ker f\rightarrow \Ext^n_{R}(R/\fa,M)\overset{f}\rightarrow \Hom_{R}(R/\fa,\lc^{n}_{\fa}(M))\rightarrow \Coker f\rightarrow 0$.
\end{center}
(a) By \ref{P:natmap} (a), $\Ker f$ is weakly Laskerian.
\begin{flushleft}
(b) By \ref{P:natmap} (b), $\Coker f$ is weakly Laskerian.
\end{flushleft}
\end{proof}

\section{Asymptotic behaviours}

Assume that $R = \underset{i\geq{0}}{\bigoplus}{R_i}$ is a homogeneous graded noetherian ring and $M = \underset{i\in \Z}\bigoplus{M_i}$ is
a graded $R$--module. The module $M$ is said to have

(a) the property of {\it asymptotic stablity of associated primes} if there exists an integer $n_0$ such that $\Ass_{R_0}(M_n)=\Ass_{R_0}(M_{n_0})$
for all $n\leq{n_0}$.

(b) the property of {\it asymptotic stablity of supports} if there exists an integer $n_0$ such that $\Supp_{R_0}(M_n)=\Supp_{R_0}(M_{n_0})$ for
all $n\leq{n_0}$.

(c) the property of {\it tameness} if $M_i=0$ for all $i<<0$ or else $M_i\neq 0$ for all $i<<0$.

Note that if $M$ has the property of (a) then $M$ has the property of (b) and if $M$ has the property of (b)  then $M$ is tame i.e. has the property
of (c).

Let $s$ be a non-negative integer. In order to refine and complete \cite[Theorem 4.4]{DN} we study the asymptotic behaviour of the kernel and cokernel
of the natural homomorphism between the graded modules $\Ext_{R}^{s}(R/R_{+}, M)$ and $\Hom_{R}(R/R_{+},\lc_{R_{+}}^{s}(M))$. For more information see
also \cite{Br}.

\begin{defn}
A graded module $M$ over a homogeneous graded ring $R$ is called asymptotically zero if $M_i=0$ for all $i<<0$.
\end{defn}
All finite graded $R$--modules are asymptotically zero.

\begin{thm}\label{T:graded}
Assume that $M$ is a graded $R$--module and let $s$ be a fixed non-negative integer such that the modules
\begin{center}
$\Ext_{R}^{s-j}(R/R_{+},\lc_{R_{+}}^{j}(M))$ and  $\Ext_{R}^{s-j+1}(R/R_{+},\lc_{R_{+}}^{j}(M))$ , $0\leq j<s$
\end{center}
are asymptotically zero {\rm{(}}e.g. they might be finite{\rm{)}}. Then the kernel and the cokernel of the natural homomorphism

$$f:\Ext_{R}^{s}(R/R_{+}, M)\lo \Hom_{R}(R/R_{+},\lc_{R_{+}}^{s}(M))$$
are asymptotically zero. Therefore $\Ext_{R}^{s}(R/R_{+}, M)$ has one of the properties of {\rm{(}}a{\rm{)}}, {\rm{(}}b{\rm{)}} and {\rm{(}}c{\rm{)}}
if and only if $\Hom_{R}(R/R_{+},\lc_{R_{+}}^{s}(M))$ has.

\end{thm}
\begin{proof}
  Note that the class of asymptotically zero graded modules is Serre
  subcategory of the category of graded $R$--modules and use
  \ref{P:natmap} (c).
\end{proof}

\providecommand{\bysame}{\leavevmode\hbox
to3em{\hrulefill}\thinspace}

\end{document}